\documentclass[12pt]{article}
\usepackage{amsmath}
\usepackage{amsfonts}
\usepackage{amssymb}
\usepackage{color}

\usepackage{graphicx}

\newtheorem{theorem}{Theorem}[section]

\newtheorem{corollary}[theorem]{Corollary}

\newtheorem{definition}[theorem]{Definition}
\newtheorem{example}[theorem]{Example}

\newtheorem{lemma}[theorem]{Lemma}

\newtheorem{proposition}[theorem]{Proposition}
\newtheorem{remark}[theorem]{Remark}

\newenvironment{proof}[1][Proof]{\noindent\textbf{#1.} }{\ \rule{0.5em}{0.5em}}
\usepackage{color}
\definecolor{sepia}{cmyk}{0, 0.83, 1, 0.70}
\usepackage[plainpages=false,pdfpagelabels,bookmarksnumbered,%
 colorlinks=true,%
 linkcolor=sepia,%
 citecolor=sepia,%
 unicode]{hyperref}

\begin{document}

\title{\vspace{-3em}\textbf{Regularity of the Product of Two Relaxed Cutters with Relaxation Parameters Beyond Two}}
\author{Andrzej Cegielski\thanks{%
Institute of Mathematics, University of Zielona G\'{o}ra, ul. Szafrana 4a, 65-516, Zielona G\'{o}ra, Poland, \newline
e-mail: \texttt{a.cegielski@im.uz.zgora.pl}}, \
Simeon Reich\thanks{%
Department of Mathematics, The Technion -- Israel Institute of Technology, 32000
Haifa, Israel, \newline
e-mail: \texttt{sreich@technion.ac.il}} \ and \
Rafa\l\ Zalas\thanks{%
Institute of Mathematics, University of Zielona G\'{o}ra, ul. Szafrana 4a, 65-516, Zielona G\'{o}ra Poland and Department of Mathematics, The Technion -- Israel Institute of Technology, 32000
Haifa, Israel, \newline
e-mail: \texttt{r.zalas@im.uz.zgora.pl}}}
\maketitle

\begin{abstract}
We study the product of two relaxed cutters having a common fixed point. We assume that one of the relaxation parameters is greater than two so that the corresponding relaxed cutter is no longer quasi-nonexpansive, but rather demicontractive. We show that if both of the operators are (weakly/linearly) regular, then under certain conditions, the resulting product inherits the same type of regularity. We then apply these results to proving convergence in the weak, norm and linear sense of algorithms that employ such products.



\end{abstract}

\section{Introduction}
The \emph{common fixed point problem} for two operators $T, U \colon \mathcal H \to \mathcal H$  defined on a real Hilbert space $\mathcal H$ is to
\begin{equation}\label{int:CFPP}
  \text{find a point in } \operatorname{Fix}T \cap \operatorname{Fix} U.
\end{equation}
We assume here that such a point exists and that $T$ and $U$ are $\lambda$ and $\mu$ relaxations of cutters, respectively, where $\lambda, \mu > 0$. Moreover, for the reasons explained below, we will also assume that $\lambda \mu < 4$. For the definition of a cutter see Section \ref{sec:preliminaries}.

The setting of problem \eqref{int:CFPP} is quite general and allows us to encompass various optimization problems. An example of such a situation is when $T$ and $U$ are $\lambda$ and $\mu$ relaxations of some metric projections, that is, when $T = (1-\lambda)I + \lambda P_A$ and $U = (1-\mu)I + \mu P_B$, where $A$ and $B$ are closed and convex subsets of $\mathcal H$. In this case \eqref{int:CFPP} becomes the convex feasibility problem  which is to find a point in $A \cap B$.  Another general example is when $T$ and $U$ are  averaged or even  conically averaged operators \cite{BDP22}. Note, however, that  in our setting of the common fixed point problem, the operators  $T$ and $U$ do not have to be nonexpansive.

A prototypical iterative method for solving problem \eqref{int:CFPP} is defined as follows:
\begin{equation}\label{int:xk}
  x^0 \in \mathcal H, \qquad x^{k+1} := x^k + \alpha_k (UT(x^k) - x^k), \qquad k = 0, 1,2, \ldots,
\end{equation}
where $\alpha_k \in [\varepsilon,1- \varepsilon]$ and $\varepsilon > 0$. Iteration \eqref{int:xk} has its roots in the method of alternating projections and the Krasnosel'ski\u{\i}--Mann method. The first question that we want to address in this paper is as follows:

\begin{center}
    \emph{(Q1) What conditions guarantee weak, norm and linear convergence of method \eqref{int:xk}?}
\end{center}

Note that for $\lambda, \mu \in (0,2)$, the product $UT$ is also a $\nu$-relaxed cutter for some $\nu \in (0,2)$ and that $\operatorname{Fix}UT = \operatorname{Fix} U  \cap \operatorname{Fix} T$; see \cite[Proposition 1(d)]{YO04} or \cite[Theorem 2.1.46]{Ceg12}. It  has only recently been  shown in \cite[Theorem 3.8]{Ceg23} that $UT$ enjoys analogous properties when $\lambda \mu < 4$; see Theorem \ref{t-comp} below. The latter inequality allows one of the relaxation parameters to be greater than two. However, in such a case $\nu > 2$ and $UT$ is no longer quasi-nonexpansive, but rather demicontractive, see Section \ref{sec:RelCutters}. A result closely related to \cite{Ceg23} has been established in \cite{BDP22} for conically averaged operators. We comment on the relation between \cite{Ceg23} and \cite{BDP22} in Remark \ref{rem:BDP22}.

One of the possible answers to question (Q1) has been presented in \cite[Theorem 6.1]{CRZ18}. Roughly speaking, (weak/linear) regularity of the operators $T$ and $U$, when combined with  some additional assumptions,  implies (weak/linear) convergence of method \eqref{int:xk}. We comment briefly on some of these regularities below; see also Section \ref{sec:regular_operators}. There are quite a few results in the literature which fall into the above-mentioned pattern. In particular, for the linear convergence, see \cite{BRZ18, BKRZ19, BB96, BLT17, BNP15, CRZ20, CET21, KRZ17, WHLY17, XC21, ZNLY18}, to name  but a few.  Note, however, that the arguments used in \cite{CRZ18}, as well as those used in the cited works, directly or indirectly, require that $\lambda, \mu \in (0,2)$. Can we thus expect that  a statement similar to the one above continues  to be true when $\lambda \mu < 4$? A positive answer was given in \cite{Ceg23},  but  only in the case of weak convergence.  To the best of our knowledge, the parts of question (Q1) regarding norm and linear convergence have not been considered when $\lambda \mu < 4$.  We return to this below.

Weak regularity of $T$ is oftentimes phrased as the demiclosedness  of the operator $I - T$ at $0$. This is also known as the demiclosedness principle. Such a property is a quite common assumption in the study of fixed point iterations \cite{BC17, Ceg12, Mar77}. Some extensions of this property can be found, for example, in \cite{BRZ19, Bau13, BC01, BCP20, Ceg15}.
Linear regularity of $T$ can  also be  expressed as the metric subregularity of $I - T$ at  a certain  point; see \cite[Definition 3.17]{Iof16} and Remark \ref{rem:MetSub}. Metric subregularity plays an important role in the study of linear convergence of set-valued fixed point iterations \cite{LTT20, LTT18a}. We note that there  also are  other notions of regularity such as H\"{o}lder regularity \cite{BLT17, CET21} and  gauge  regularity \cite{LTT18b}, both of which guarantee norm convergence with some error estimates.

The usefulness of the above-mentioned regularities suggests that a more careful study should be devoted to these properties. This bring us to the main question that we want to address in this paper:

\begin{center}
    \emph{(Q2) Does the (weak/linear) regularity of the individual operators $U$ and $T$ imply the (weak/linear) regularity of the product $UT$?}
\end{center}

The answer to question (Q2) for all three types of regularity is again positive when $\lambda, \mu \in (0,2)$ \cite[Corollary 5.6]{CRZ18}. Analogous results  addressing  some of these regularities can be found in \cite{BRZ18, BKRZ19, Ceg15, CRZ20, CZ14, RZ16}. See also \cite{CET21} for H\"{o}lder regularity. It should  also be noted that \cite{Ceg23} answered question (Q2) for weak regularity when $\lambda \mu < 4$.  To the best of our knowledge, the parts of question (Q2) regarding regularity and linear regularity remain open when $\lambda  \mu < 4$.

The main contribution of  the present  paper is to positively answer questions (Q1) and (Q2) raised above assuming that $\lambda \mu < 4$. In particular, in Theorem \ref{thm:main}, which is the main result of this paper, we answer question (Q2). As a consequence, in Theorem \ref{thm:main2} we also answer question (Q1). In Corollaries \ref{cor:main3} and \ref{cor:main4} we specialized these theorems to orthogonal projections. We also introduced a new technical Lemma \ref{lem:LB2}. This allowed us not only to prove our results in a systematic way, but also to re-establish some results from \cite{Ceg23} concerning weak regularity  and weak convergence.

The organization of  our  paper is as follows: in Section \ref{sec:preliminaries} we recall some properties of relaxed cutters which we apply in the subsequent sections. In particular, we connect them to demicontractions and we elaborate on their products. Moreover, we recall the notions of (weakly/linearly) regular operators,  where we give some basic examples.  The two above-mentioned main results, including Lemma \ref{lem:LB2}, are presented in Sections \ref{sec:main} and \ref{sec:application}.

\section{Preliminaries\label{sec:preliminaries}}
In this section we introduce notations and recall some well-known facts which we use in this paper.

In the whole paper $\mathcal{H}$ denotes a real Hilbert space with inner
product $\langle \cdot ,\cdot \rangle $ and induced norm $\Vert \cdot \Vert $.  We denote the weak convergence of a sequence $\{x^k\}_{k=0}^\infty$ to a point $x \in \mathcal H$ by $x^k \rightharpoonup x$ as $k \to \infty$. The distance of a point $x \in \mathcal H$ to a nonempty subset $A \subseteq\mathcal H$ is given by $d(x, A) := \inf_{a \in A} \|x - a\|$.

Let $T:\mathcal{H}\rightarrow \mathcal{H}$. An operator $T_{\lambda }:=\operatorname{Id}+\lambda (T-\operatorname{Id})$ is called a $\lambda$\textit{-relaxation} of $T$, where $\lambda \geq 0$.  Note that $T_\lambda - \operatorname{Id} = \lambda (T - \operatorname{Id})$ and $(T_{\lambda })_{\mu}=T_{\lambda \mu }$, where $\lambda ,\mu \geq 0$. In particular,  $T = (T_{\frac 1 \lambda})_\lambda$ and  $T-\operatorname{Id}=\lambda (T_{\frac 1 \lambda}-\operatorname{Id})$, where $\lambda >0$.

The subset $\operatorname{Fix}T:=\{x\in \mathcal{H}:T(x)=x\}$ is called the \textit{fixed point set} of $T$ and  each one if  its elements is called a \textit{fixed point}. Clearly, $\operatorname{Fix} T_\lambda = \operatorname{Fix}T$ for $\lambda > 0$.

\subsection[Relaxed Cutters]{\label{sec:RelCutters}Relaxed Cutters}

\begin{definition}
\label{def:cutter}%
\rm\ %
An operator $T \colon \mathcal{H}\to \mathcal{H}$ with $\operatorname{Fix}T\neq
\emptyset $ is called a \textit{cutter} if for all $x\in \mathcal{H}$ and for all $z\in \operatorname{Fix}T$, we have
\begin{equation}
\langle z-T(x),x-T(x)\rangle \leq 0.
\end{equation}%
A $\lambda $-relaxation of a cutter, where $\lambda \geq 0$, is called a $\lambda $-\textit{relaxed cutter}.
\end{definition}

In the literature a cutter is also called a $\mathcal{T}$-class operator
\cite{BC01}, a firmly quasi-nonexpansive operator \cite{YO04, BC17} or a directed
operator \cite{CS09}. We use the  term  ``cutter'' following \cite{CC10, Ceg12}.

We relate relaxed cutters to quasi-nonexpansive operators and to demicontractions. First we recall the two aforementioned definitions.
\begin{definition}
\label{d-QNE}%
\rm\ %
We say that an operator $T \colon \mathcal{H}\to \mathcal{H}$ with $\operatorname{Fix}T\neq
\emptyset$ is
\begin{enumerate}
  \item[(a)]  \textit{quasi-nonexpansive }(QNE), if for all $x\in \mathcal{H}$ and all $z\in \operatorname{Fix}T$, we have
      \begin{equation}
        \Vert T(x)-z\Vert\leq \Vert x-z\Vert;
      \end{equation}
  \item[(b)]  $\rho $\textit{-demicontractive} ($\rho $-DC), where $\rho \in ( -\infty ,1)$, if for all $x\in \mathcal{H}$ and all $z\in \operatorname{Fix}T$, we have
      \begin{equation}\label{d-DC:eq}
        \Vert T(x)-z\Vert ^{2}\leq \Vert x-z\Vert ^{2}+\rho \Vert T(x)-x\Vert ^{2}.
      \end{equation}
\end{enumerate}
\end{definition}

\begin{proposition} \label{prop:demicontractive}
  Let $T \colon \mathcal{H}\to \mathcal{H}$ and assume that $\operatorname{Fix}T\neq \emptyset$. The following conditions are equivalent:
  \begin{enumerate}
    \item[(i)] $T$ is  a  $\lambda$-relaxation of some cutter, where $\lambda >0$;
    \item[(ii)] $T$ is  a  $\theta$-relaxation of some quasi-nonexpansive operator, where $\theta = \frac{\lambda}{2} > 0$;
    \item[(iii)] $T$ is $\rho$-demicontractive, where $\rho = \frac{\lambda - 2}{\lambda} < 1$.
  \end{enumerate}
\end{proposition}

\begin{proof}
    See \cite[Remark 3.1]{Mai08} or \cite[Theorem 2.5]{Ceg23}. See also  \cite[Theorem 2.1.39]{Ceg12} for $\lambda \in (0,2]$.
\end{proof}

\begin{remark}
\rm\ %
  The notion of  a  demicontraction in a Hilbert space $\mathcal{H}$ was introduced by Hicks and Kubicek in 1977 \cite{HK77}. M\u{a}ru\c{s}ter \cite{Mar77} introduced demicontractions using  a  different property and called it condition (A); see \eqref{prop:condA:eq} below. Moore \cite{Moo98} proved that these two notions, \eqref{d-DC:eq} and \eqref{prop:condA:eq}, are equivalent.
\end{remark}

\begin{remark}
\rm\ %
  If $T$ is $\rho$-demicontractive for some $\rho \in (-\infty, 0)$, then this operator is also called $(-\rho)$-\emph{strongly quasi-nonexpansive}; see \cite{Ceg12}. In view of Proposition \ref{prop:demicontractive}, strongly quasi-nonexpansive operators are $\lambda$-relaxations of cutters, where $\lambda \in (0,2)$. Equivalently, strongly quasi-nonexpansive operators are $\theta$-relaxations of quasi-nonexpansive operators, where $\theta \in (0,1)$.
\end{remark}

 We provide a few examples of cutters and relaxed cutters. We begin with three very broad classes of operators which are well established in the optimization area.

\begin{example}[Averaged Operators] \rm
  Recall that $T \colon \mathcal H \to \mathcal H$ is \emph{averaged} if it is a $\theta$-relaxation of a nonexpansive operator for some $\theta \in (0, 1)$, that is, $T = (1-\theta) I + \theta N$, where $\|N(x) - N(y)\| \leq \|x - y\|$ for all $x,y \in \mathcal H$. It is easy to see that if $\operatorname{Fix} T \neq \emptyset$, then $N$ is quasi-nonexpansive and $T$ is a $\lambda$-relaxed cutter, where $\lambda = 2\theta$.
\end{example}

\begin{example}[Firmly Nonexpansive Operators] \rm
  Recall that $T \colon \mathcal H \to \mathcal H$ is \emph{firmly nonexpansive} if $\langle T(x) - T(y), x - y \rangle \geq \|T(x) - T(y)\|^2$ for all $x,y \in \mathcal H$. It can be shown that if $T$ is firmly nonexpasive, then $T$ is $\frac 1 2$-averaged; see, for example, \cite[Theorem 2.2.10]{Ceg12}. Thus, if $\operatorname{Fix} T \neq \emptyset$, then $T$ is a cutter.
\end{example}

\begin{example}[Strict Contraction] \label{ex:contraction} \rm
  Let $T \colon \mathcal H \to \mathcal H$ and assume that $T$ is an $\alpha$-strict contraction for some $\alpha \in (0,1)$, that is, $\|T(x) - T(y)\| \leq \alpha \|x - y\|$ for all $x,y \in \mathcal H$. Then $T$ is $\theta$-averaged, where $\theta = \frac{\alpha+1}{2}$; see, for example, \cite[Theorem 2.2.34]{Ceg12} or \cite[Proposition 4.38]{BC17}. In particular, $T$ is a $\lambda$-relaxed cutter, where $\lambda = \alpha+1$.
\end{example}

We now proceed to specific examples.

\begin{example}[Metric Projection] \rm
  Let $C\subseteq \mathcal H$ be nonempty, closed and convex. The \emph{metric projection} $P_C(x):=\operatorname{argmin}_{z\in C}\|z-x\|$ is firmly nonexpansive and  $\operatorname{Fix} P_C = C$; see, for example, \cite[Theorem 2.2.21]{Ceg12}. Thus any metric projection is a cutter.
\end{example}

\begin{example}[Subgradient Projection]\label{ex:subProj} \rm
  Let $f\colon \mathcal{H}\to \mathbb{R}$ be a lower semicontinuous and convex function with nonempty sublevel set $S(f,0):=\{x\in \mathcal{H}\colon f(x)\leq 0\}\neq \emptyset $. For each $x\in \mathcal{H}$, let  $g(x)$  be a chosen subgradient from the subdifferential set $\partial f(x):=\{g\in \mathcal{H}\colon f(y)\geq f(x)+\langle g,y-x\rangle \text{ for all }y\in \mathcal{H}\}$,  which, by \cite[Proposition 16.27]{BC17}, is nonempty.  Note here that we apply \cite[Proposition 16.27]{BC17} to the real-valued functional $f$.  The  \emph{subgradient projection}
  \begin{equation}\label{ex:subProj:eq}
    P_{f}(x):=
    \begin{cases}
      x-\frac{f(x)}{\|g(x)\|^2}g(x), & \mbox{if } f(x)>0 \\
      x, & \mbox{otherwise,}
    \end{cases}
  \end{equation}
  is a cutter  and $\operatorname{Fix} P_{f}=S(f,0)$; see, for example, \cite[Corollary 4.2.6]{Ceg12}.
\end{example}

\begin{example}[Proximal Operator] \label{ex:proximalOp} \rm
  Let $f\colon\mathcal H \to \mathbb R \cup\{+\infty\}$ be a proper lower semicontinuous and convex function. The \emph{proximal operator}  $\operatorname{prox}\nolimits_f (x) := \operatorname{argmin}_{y\in \mathcal H} ( f(y)+\frac 1 2 \|y-x\|^2)$  is firmly nonexpansive and $\operatorname{Fix} (\operatorname{prox}_f)=\operatorname{Argmin}_{x\in\mathcal H} f(x)$; see \cite[Propositions 12.28 and 12.29]{BC17}. Thus if $f$ has at least one minimizer, then $\operatorname{prox}_f$ is a cutter; see \cite[Theorem 2.2.5]{Ceg12}.
\end{example}

\subsection{Product of Two Relaxed Cutters}

Below is a very simple but useful characterization of relaxed cutters; see \cite[Remark 2.1.31]{Ceg12},  where the case $\lambda \in (0,2]$ is considered.
\begin{proposition}\label{prop:condA}
  Let $T \colon \mathcal H \to \mathcal H$ be an operator with $\operatorname{Fix}T \neq \emptyset$ and let $\lambda > 0$. Then $T$ is  a  $\lambda$-relaxed cutter if and only if for all $x\in \mathcal{H}$ and all $z\in \operatorname{Fix}T$, we have
    \begin{equation}\label{prop:condA:eq}
        \lambda \langle z-x,T(x)-x\rangle \geq \Vert T(x)-x\Vert ^{2}.
    \end{equation}%
\end{proposition}

\begin{proof}
  Put $U := T_{\frac 1 \lambda}$ so that $T = U_{\lambda}$ and $\operatorname{Fix} T = \operatorname{Fix} U$. Moreover, let $x \in \mathcal H$ and $z \in \operatorname{Fix} T$. Then,
  \begin{equation}\label{pr:prop:condA:eq}
    \lambda \langle z-x, T(x)-x\rangle - \|T(x)-x\|^2 = - \lambda^2 \langle z-U(x), x - U(x)\rangle.
  \end{equation}
  If $T$ is  a  $\lambda$-relaxed cutter, then $U = T_{\frac 1 \lambda}$ is a cutter. Thus the right-hand side of \eqref{pr:prop:condA:eq} is  non-negative  and we arrive at \eqref{prop:condA:eq}. On the other hand, when \eqref{prop:condA:eq} holds, then the left-hand side of \eqref{pr:prop:condA:eq} is  non-negative,  which shows that $U$ is a cutter.
\end{proof}

\bigskip
 We now  recall a technical lemma from \cite[Lemma 3.5]{Ceg23}.

\begin{lemma} \label{lem:ni}
  Let $\lambda, \mu >0$ and assume that $\lambda \mu < 4$. Then
  \begin{equation}\label{lem:ni:eq1}
    \nu :=\frac{4(\lambda +\mu -\lambda \mu )}{4-\lambda \mu }\
    \begin{cases}
      < 2, & \text{if } \max\{\lambda, \mu\} < 2\\
      = 2, & \text{if } \max\{\lambda, \mu\} = 2\\
      > 2, & \text{if } \max\{\lambda, \mu\} > 2.
    \end{cases}
  \end{equation}
  Moreover,
  \begin{equation}\label{lem:ni:max}
    \nu \geq \max\{\lambda, \mu\}
  \end{equation}
  and $\nu$ satisfies
  \begin{equation}\label{lem:ni:eq2}
    4 \left(\frac 1 \lambda - \frac 1 \nu \right) \left(\frac 1 \mu - \frac 1 \nu \right)
    = \left( 1 - \frac 2 \nu\right)^2.
  \end{equation}
\end{lemma}

The following lemma can be deduced from the proof of \cite[Theorem 3.8]{Ceg23}. Since we  rely  heavily on this result in the sequel, we sketch its proof for completeness.

\begin{lemma} \label{lem:LB1}
  Let $T \colon \mathcal H \to \mathcal H$ and $U \colon \mathcal H \to \mathcal H$ be $\lambda$- and $\mu$-relaxed cutters, respectively, where $\lambda, \mu >0$. Assume that $\lambda \mu < 4$ and $\operatorname{Fix}T\cap \operatorname{Fix}U\neq \emptyset$. Moreover, let $\nu$ be defined by \eqref{lem:ni:eq1}. Then, for each $x \in \mathcal H$ and $z \in \operatorname{Fix}T \cap \operatorname{Fix}U$, we have
  \begin{equation} \label{lem:LB1:ineq}
    \langle z-x, UT(x)-x\rangle
    \geq \left\| \sqrt{\frac{1}{\lambda }-\frac{1}{\nu }}\Big(T(x)-x\Big)
    \pm \sqrt{\frac{1}{\mu }-\frac{1}{\nu }} \Big(UT(x)-T(x)\Big)\right\|^{2}\geq 0\text{,}
  \end{equation}
  where the sign ``$\pm $'' should be replaced by ``$+$'' when $\max \{\lambda, \mu\} \geq 2$ and by ``$-$'' when $\max \{\lambda, \mu\} < 2$.
\end{lemma}

\begin{proof}
  In order to shorten the notation, put $a:=T(x)-x$, $b:=UT(x)-T(x)$ and $c:=UT(x)-x$. Note that $a+b = c$. Using Proposition \ref{prop:condA}, we have $\langle z-x,a\rangle \geq \frac 1 \lambda \|a\|^2$ and $\langle z-T(x),b\rangle \geq \mu \|b\|^2$. Consequently,
  \begin{align}\label{pr:lem:LB1:ineq1}
    \langle z-x,c\rangle -\frac{1}{\nu }\|c\|^2
    & \textstyle = \langle z-x,a\rangle + \langle z-T(x),b\rangle + \langle a, b\rangle -\frac{1}{\nu }\|a+b\|^2 \nonumber \\
    & \textstyle \geq \frac 1 \lambda \|a\|^2 + \frac 1 \mu \|b\|^2 + \langle a, b\rangle -\frac{1}{\nu }\|a+b\|^2 \nonumber \\
    & \textstyle = \left(\frac 1 \lambda - \frac 1 \nu \right) \|a\|^2
      + \left(\frac 1 \mu - \frac 1 \nu \right) \|b\|^2
      + \left( 1 - \frac 2 \nu\right) \langle a, b\rangle.
  \end{align}
  Assume now that $\max\{\lambda, \mu\} \geq 2$. Then,  using  \eqref{lem:ni:eq1}--\eqref{lem:ni:eq2}, we get $1 - \frac 2 \nu \geq 0$ and
  \begin{equation}\label{} \textstyle
    2 \sqrt{\frac 1 \lambda - \frac 1 \nu} \cdot \sqrt{\frac 1 \mu - \frac 1 \nu}
    =  1 - \frac 2 \nu.
  \end{equation}
  Consequently,
  \begin{equation}\label{pr:lem:LB1:ineq2}
    \textstyle \left(\frac 1 \lambda - \frac 1 \nu \right) \|a\|^2
      + \left(\frac 1 \mu - \frac 1 \nu \right) \|b\|^2
      + \left( 1 - \frac 2 \nu\right) \langle a, b\rangle
      = \left\| \sqrt{\frac{1}{\lambda }-\frac{1}{\nu }}a
      + \sqrt{\frac{1}{\mu }-\frac{1}{\nu }} b\right\|^2.
  \end{equation}
  Assume  next  that $\max\{\lambda, \mu\} < 2$. Then,  by again using  \eqref{lem:ni:eq1}--\eqref{lem:ni:eq2}, we get $1 - \frac 2 \nu < 0$ and
  \begin{equation}\label{} \textstyle
    2 \sqrt{\frac 1 \lambda - \frac 1 \nu} \cdot \sqrt{\frac 1 \mu - \frac 1 \nu}
    =  - \left(1 - \frac 2 \nu \right).
  \end{equation}
  Consequently, in this case we get
  \begin{equation}\label{pr:lem:LB1:ineq3}
    \textstyle \left(\frac 1 \lambda - \frac 1 \nu \right) \|a\|^2
      + \left(\frac 1 \mu - \frac 1 \nu \right) \|b\|^2
      + \left( 1 - \frac 2 \nu\right) \langle a, b\rangle
      = \left\| \sqrt{\frac{1}{\lambda }-\frac{1}{\nu }}a
      - \sqrt{\frac{1}{\mu }-\frac{1}{\nu }} b\right\|^2.
  \end{equation}
  By combining \eqref{pr:lem:LB1:ineq1}, \eqref{pr:lem:LB1:ineq2} and \eqref{pr:lem:LB1:ineq3}, we arrive at
  \begin{equation} \label{pr:lem:LB1:ineq}
    \textstyle \langle z-x,c\rangle -\frac{1}{\nu }\|c\|^{2}
    \geq \left\| \sqrt{\frac{1}{\lambda }-\frac{1}{\nu }}a
    \pm \sqrt{\frac{1}{\mu }-\frac{1}{\nu }} b\right\|^{2}\geq 0,
  \end{equation}
  where the sign ``$\pm $'' should be replaced by ``$+$'' when $\max \{\lambda, \mu\} \geq 2$ and by ``$-$'' when $\max \{\lambda, \mu\} < 2$. This completes the proof.
\end{proof}

\bigskip
An extended form of the following theorem was proved in
\cite[Theorem 3.8]{Ceg23}.

\begin{theorem}
\label{t-comp}
 Let $T \colon \mathcal H \to \mathcal H$ and $U \colon \mathcal H \to \mathcal H$ be $\lambda$- and $\mu$-relaxed cutters, respectively, where $\lambda, \mu >0$. Assume that $\lambda \mu < 4$ and $\operatorname{Fix}T\cap \operatorname{Fix}U\neq \emptyset$. Then the product $UT$ is  a  $\nu$-relaxed cutter with $\operatorname{Fix} UT = \operatorname{Fix}T\cap \operatorname{Fix}U$, where $\nu$ is defined by \eqref{lem:ni:eq1}. Consequently,  $(UT)_{1/\nu }$ is a cutter.
\end{theorem}

\begin{remark}\label{rem:BDP22} \rm
  Recall that a conically averaged operator is a $\theta$-relaxation of a nonexpansive operator, where $\theta > 0$ \cite[Definition 2.1]{BDP22}. The result of \cite[Proposition 2.15]{BDP22} shows that if $T$ and $U$ are $\theta_1$- and $\theta_2$-conically averaged and $\theta_1 \theta_2 < 1$, then $TU$ is $\theta$-conically averaged with
      \begin{equation}\label{}
        \theta = \frac{\theta_1 + \theta_2 - 2 \theta_1 \theta_2}{1 - \theta_1 \theta_2}.
      \end{equation}
  On the other hand, it is not difficult to see that $T$ and $U$ are $\lambda = 2\theta_1$ and $\mu = 2\theta_2$ relaxations of some firmly nonexpansive operators. Moreover, firmly nonexpansive operators having  a  fixed point are cutters. Thus the inequality $\theta_1 \theta_2 < 1$ translates  into  $\lambda \mu < 4$ and $\nu$ of Theorem \ref{thm:main} satisfies $\nu = 2 \theta$.
\end{remark}

\subsection[Regular Sets]{\label{sec:regular_sets}Regular Sets}

\begin{definition}
\label{d-SR}%
\rm\ %
Let $A,B\subseteq \mathcal{H}$ be  closed and convex, and  such that $A\cap B\neq \emptyset $. Let
$S\subseteq \mathcal{H}$ be nonempty. We say that the  pair  $\{A,B\}$ is:

\begin{enumerate}
\item[(a)] \emph{regular} over $S$ if for any sequence $\{x^{k}\}_{k=0}^{\infty
}\subseteq S$, we have%
\begin{equation}
\lim_{k \to \infty} \max \{d(x^{k},A),d(x^{k},B)\}=0 \quad \Longrightarrow \quad
\lim_{k \to \infty}  d(x^{k},A \cap B)=0\text{;}
\end{equation}
\item[(b)] \emph{$\kappa $-linearly regular} over $S$, where $\kappa >0$, if for
for arbitrary $x\in S$, we have%
\begin{equation}
d(x, A \cap B)\leq \kappa \max \{d(x,A),d(x,B)\}\text{.}
\end{equation}
\end{enumerate}

If any of the above regularity conditions holds for $S=\mathcal{H}$, then we
omit the phrase \textquotedblleft over $S$\textquotedblright. If any of the
above regularity conditions holds for every bounded subset $S\subseteq
\mathcal{H}$, then we precede the corresponding term with the adverb
\textquotedblleft boundedly\textquotedblright\ while omitting the phrase
\textquotedblleft over $S$\textquotedblright\ (we allow $\kappa $ to depend
on $S$ in (b)).
\end{definition}

Below we list a few known examples of regular pairs of sets. For an extended list with more than two sets, see \cite{BB96} or \cite{BNP15}.

\begin{example} \rm \label{ex:RegSets}
Let $A,B\subseteq \mathcal{H}$ be closed and convex, and such that $A\cap B\neq \emptyset $.

\begin{enumerate}
\item[$\mathrm{(i)}$] If $\dim \mathcal{H}<\infty $, then $\{A, B\}$ is boundedly regular;

\item[$\mathrm{(ii)}$] If $A$ and $B$ are half-spaces, then $\{A, B\}$ is linearly regular;

\item[$\mathrm{(iii)}$] If $A \cap \operatorname{int}(B)\neq \emptyset$, then $\{A, B\}$ is boundedly linearly regular;

\item[$\mathrm{(iv)}$] If $\dim \mathcal{H}<\infty $, $A$ is a half-space and $A\cap \operatorname{ri}B \neq \emptyset $, then $\{A, B\}$ is boundedly linearly regular.

\item [$\mathrm{(v)}$] If $\dim \mathcal{H}<\infty $ and  the pair $\{A, B\}$ is transversal, that is, when $\operatorname{ri}A \cap \operatorname{ri} B \neq \emptyset$,  then $\{A, B\}$ is boundedly linearly regular.

\end{enumerate}
\end{example}

\subsection[Regular Operators]{\label{sec:regular_operators}Regular Operators}

The following definition can be found in \cite[Definition 3.1]{CRZ18}.

\begin{definition}
\label{d-R}%
\rm\ %
Let $S\subseteq \mathcal{H}$ be nonempty. We say that an operator $T\colon
\mathcal{H}\rightarrow \mathcal{H}$ is:

\begin{enumerate}
\item[(a)] \textit{weakly regular} over $S$ if for any sequence $%
\{x^{k}\}_{k=0}^{\infty }\subseteq S$, we have
\begin{equation}
\left( x^{k}\rightharpoonup y\text{ and }\Vert T(x^{k})-x^{k}\Vert
\rightarrow 0\text{ as }k\rightarrow \infty \right) \quad \Longrightarrow
\quad y\in \operatorname{Fix}T\text{;}  \label{e-WR}
\end{equation}

\item[(b)] \textit{regular} over $S$ if for any sequence $%
\{x^{k}\}_{k=0}^{\infty }\subseteq S$, we have
\begin{equation}
\lim_{k\rightarrow \infty }\Vert T(x^{k})-x^{k}\Vert =0\quad \Longrightarrow
\quad \lim_{k\rightarrow \infty }d(x^{k},\operatorname{Fix}T)=0\text{;}
\label{e-R}
\end{equation}

\item[(c)] $\delta $-\textit{linearly regular} over $S$, where $\delta >0$,
if for arbitrary $x\in S$, we have
\begin{equation}
\Vert T(x)-x\Vert \geq \delta d(x,\operatorname{Fix}T)\text{.}  \label{e-LR}
\end{equation}%
The constant $\delta $ is called a \textit{modulus} of the linear regularity%
\textit{\ of }$T$ over $S$.
\end{enumerate}
If any of the above regularity conditions holds for $S=\mathcal{H}$, then we
omit the phrase \textquotedblleft over $S$\textquotedblright . If any of the
above regularity conditions holds for every bounded subset $S\subseteq
\mathcal{H}$, then we precede the corresponding term with the adverb
\textquotedblleft boundedly\textquotedblright\ while omitting the phrase
\textquotedblleft over $S$\textquotedblright\ (we allow $\delta $ to depend
on $S$ in (c)).
\end{definition}

\begin{remark}[Demiclosedness Principle] \label{rem:DCP} \rm
  Note that the condition ``$I - T$ is demiclosed at $0$'' is an equivalent way of saying that the operator $T$ is weakly regular. A weakly regular operator is also called an operator satisfying the \textit{demiclosedness principle}; see, e.g., \cite[Definition 4.2]{Ceg15}. The latter notion is usually applied to nonexpansive operators  which are weakly regular;  see, e.g., \cite[Theorem 4.27]{BC17}.
\end{remark}

\begin{remark}[Metric Subregularity] \label{rem:MetSub} \rm
  Observe that linear regularity of the operator $T \colon \mathcal H \to \mathcal H$ over $B(z, r)$, where $z \in \operatorname{Fix} T$ and $r > 0$, is nothing else  but the  metric subregularity of $F := I -T$ at $(z,0)$. Recall that a set-valued mapping $F \colon \mathcal H\rightrightarrows \mathcal H$ is \emph{metrically subregular} at $(\bar x, \bar y) \in \operatorname{gph} F$ if there are $\delta >0$ and $r > 0$ such that for all $x \in B(\bar x,r)$, we have
  \begin{equation}\label{}
    d(\bar y, F(x)) \geq \delta d(x, F^{-1}(\bar y));
  \end{equation}
  see, for example, \cite[Definition 3.17]{Iof16}.
\end{remark}

Below we present a few examples of regular operators with an emphasis on the linear regularity.

\begin{example}[Linear Operators] \rm
  Let $T\colon \mathcal H \to \mathcal H$ be a bounded linear operator. Then $T$ is $\delta$-linearly regular if and only if the range $\mathcal R(I - T)$ is closed.
  In such a case we can put $\delta = \inf\{\|x - T(x)\| \colon x \in (\operatorname{Fix}T)^\perp, \|x\| = 1\} > 0$. This follows directly from the closed range theorem; see, for example, \cite[Theorem 8.18]{Deu01}. In particular, when $\dim \mathcal H < \infty$, then $T$ is linearly regular and $\delta$ is the smallest positive singular value of $I - T$ \cite[Lemma 3.4]{CRZ20}.
\end{example}

\begin{example}[Polyhedral Graph] \rm
  Let $\mathcal H = \mathbb R^d$ and let $T\colon \mathcal H \to \mathcal H$ be an operator with $\operatorname{Fix}T \neq \emptyset$. Assume that the graph of $T$ is a polyhedral set, that is, it can be represented as a finite intersection of closed half-spaces and/or hyperplanes. Then $T$ is boundedly linearly regular; see \cite[Proposition 1]{CET21}.
\end{example}

\begin{example}[Strict Contraction Cont.] \rm
  Following Example \ref{ex:contraction}, let $T \colon \mathcal H \to \mathcal H$ be an $\alpha$-strict contraction for some $\alpha \in (0,1)$. Then $T$ is $(1-\theta)$-linearly regular. Indeed, knowing that $T$ has only one fixed point, say $\operatorname{Fix}T = \{z\}$, for any $x \in \mathcal H$, we get
  \begin{equation}\label{}
    \|x - T(x)\| \geq \|x - z\| - \|T(x) - T(z)\| \geq (1-\alpha)\|x-z\| = (1 - \alpha) d(x, \operatorname{Fix} T).
  \end{equation}
\end{example}

\begin{example}[Subgradient Projection Cont.] \rm
  Let $P_f$ be a subgradient projection as defined in Example \ref{ex:subProj}. Assume that $\partial f$ is uniformly bounded on bounded sets (see \cite[Proposition 7.8]{BB96}). Then the following statements hold:
    \begin{enumerate}
        \item[$\mathrm{(i)}$] \textrm{$P_{f}$ is weakly regular. }

        \item[$\mathrm{(ii)}$] \textrm{If $f$ is $\alpha $-strongly convex, where $\alpha >0$, then $P_{f}$ is boundedly regular. }

        \item[$\mathrm{(iii)}$] \textrm{If $f(z)<0$ for some $z$, then $P_{f}$ is boundedly linearly regular. }
    \end{enumerate}
    For the proof, see \cite[Example 2.11]{CRZ20}.
\end{example}

\begin{example}[Proximal Operator Cont.]\rm
  Let $f\colon\mathbb R^d \to \mathbb R \cup\{+\infty\}$ be a proper lower semicontinuous and convex function which attains its minimum. Moreover, let $f^*$ be its minimal value and denote the set of its minimizers by $M$. Let $t > 0$ and, following Example \ref{ex:proximalOp}, consider the proximal operator $\operatorname{prox}\nolimits_{tf}(x) = \operatorname{argmin}_{y\in \mathcal H} \left( f(y)+\frac 1 {2t} \|y-x\|^2 \right).$ Clearly, $\operatorname{Fix}(\operatorname{prox}_{tf}) = M$. Assume the following subdifferential error bound: there are $L >0$ and $\varepsilon >0$ such that
  \begin{equation}\label{}
    d(x, M) \leq L \cdot d(0, \partial f(x))
  \end{equation}
  for all $x \in S(f, f^* + \varepsilon) = \{x \colon f(x) \leq f^* + \varepsilon\}$. Then, in view of \cite[Theorem 3.4]{DL18}, we get
  \begin{equation}\label{}
    \|x - \operatorname{prox}\nolimits_{tf}(x)\| \geq \frac{t}{L+t} d(x, M)
  \end{equation}
  for all $x \in S(f, f^* + \varepsilon)$. In particular, the proximal operator $\operatorname{prox}_{tf}$ is $\delta$-linearly regular over the sublevel set $S(f, f^* + \varepsilon)$, where $\delta = \frac{t}{L+t}$.
\end{example}

\begin{remark} \rm
    The regularity of a pair of sets $\{A, B\}$ can in fact be introduced using the regularity of a certain operator. To see this, following \cite[Remark 2.13]{KRZ17} or \cite[Remark 3.3]{CRZ18}, consider the operator $T$ defined by
    \begin{equation}\label{}
      T(x) := \begin{cases}
                P_A(x), & \text{if } d(x, A) \geq d(x, B)\\
                P_B(x), & \text{otherwise.}
              \end{cases}
    \end{equation}
    Clearly, $T$ is a cutter and $\operatorname{Fix} T = A \cap B$. Moreover, $T$ is (linearly) regular over $S$ if and only if $\{A, B\}$ is (linearly) regular over $S$.
\end{remark}

\subsection[Fejer Monotone Sequences]{Fej\'{e}r Monotone Sequences}

\begin{definition} \rm
  We say that a sequence $\{x^{k}\}_{k=0}^{\infty }$ is \emph{Fej\'{e}r monotone} with respect to a nonempty subset $C\subseteq \mathcal{H}$ if $\Vert x^{k+1}-z\Vert \leq \Vert x^{k}-z\Vert $ for all $z\in C$ and $k=0,1,2,\ldots$.
\end{definition}

 An extended  version of the following lemma can be found in \cite[Theorem 2.16]{BB96}.

\begin{lemma} \label{lem:FM}
   Let $\{x^{k}\}_{k=0}^{\infty }\subseteq \mathcal{H}$ be Fej\'{e}r monotone with respect to  $C\subseteq \mathcal{H}$. Then:

\begin{enumerate}
\item[$\mathrm{(i)}$] $\{x^{k}\}_{k=0}^{\infty }$ converges weakly to a point $x^*\in C$ if and
only if all its weak cluster points belong to $C$;

\item[$\mathrm{(ii)}$] $\{x^{k}\}_{k=0}^{\infty }$ converges strongly to a point $x^*\in C$ if and
only if $\lim_{k}d(x^{k},C)=0$;

\item[$\mathrm{(iii)}$] If $\{x^{k}\}_{k=0}^{\infty }$ converges strongly to a point $x^*\in C$, then
\begin{equation}\label{}
  \|x^k - x^*\| \leq 2 d(x^k, C), \quad k = 0,1,2,\ldots .
\end{equation}
\end{enumerate}
\end{lemma}

\section[Main Result]{Main Result\label{sec:main}}

We begin with the following lemma.

\begin{lemma} \label{lem:LB2}
  Let $T \colon \mathcal H \to \mathcal H$ and $U \colon \mathcal H \to \mathcal H$ be $\lambda$- and $\mu$-relaxed cutters, respectively, where $\lambda, \mu >0$. Assume that $\lambda \mu < 4$ and $\operatorname{Fix}T\cap \operatorname{Fix}U\neq \emptyset$. Then, for each $x \in \mathcal H$ such that $x \notin \operatorname{Fix}T \cap \operatorname{Fix}U$, we have
  \begin{equation}\label{lem:LB2:ineq}
    \|UT(x) - x\|
    \geq \left(\frac{|\alpha|}{1+\beta \sqrt{\nu}}\right)^2
    \frac{\max \{\|T(x) - x\|^2,\ \|UT(x) - T(x)\|^2\}}
    {d(x, \operatorname{Fix}T \cap \operatorname{Fix} U)},
  \end{equation}
  where $\nu$ is defined by \eqref{lem:ni:eq1}, and where
\begin{equation}\label{lem:LB2:ab} \textstyle
  \alpha := \sqrt{\frac{1}{\lambda }-\frac{1}{\nu }}-\sqrt{\frac{1}{\mu }-\frac{1}{\nu }} \neq 0
  \quad \text{and} \quad
  \beta :=\max \left\{\sqrt{\frac{1}{\lambda }-\frac{1}{\nu }}, \sqrt{\frac{1}{\mu }-\frac{1}{\nu }} \right\} > 0.
\end{equation}
\end{lemma}

\begin{proof}
  Similarly  to what we did  before, we denote $a:=T(x)-x$, $b:=UT(x)-T(x)$ and $c:=UT(x)-x$. We  first show  that
  \begin{equation}\label{pr:lem:LB2:toShow1} \textstyle
    \left\| \sqrt{\frac{1}{\lambda }-\frac{1}{\nu }}a
    \pm \sqrt{\frac{1}{\mu }-\frac{1}{\nu }} b\right\|
    +\beta \|c\|
    \geq |\alpha| \max\{\|a\|, \|b\|\},
  \end{equation}
  where, as in Lemma \ref{lem:LB1}, the sign ``$\pm $'' should be replaced by ``$+$'' when $\max \{\lambda, \mu\} \geq 2$ and by ``$-$'' when $\max \{\lambda, \mu\} < 2$.

  Indeed, assume that $\max \{\lambda, \mu\} \geq 2$. Then the triangle inequality yields
\begin{align}
  \textstyle
  \left\| \sqrt{\frac{1}{\lambda }-\frac{1}{\nu}}a + \sqrt{\frac{1}{\mu } - \frac{1}{\nu }}b\right\|
  & = \textstyle \left\| \alpha a
  + \sqrt{\frac{1}{\mu } - \frac{1}{\nu }}(a + b) \right\| \nonumber \\
  & \textstyle \geq |\alpha| \cdot \|a\| - \sqrt{\frac{1}{\mu }-\frac{1}{\nu}}\|c\| \nonumber \\
  & \textstyle \geq |\alpha| \cdot \|a\| - \beta\|c\|.
\end{align}
Analogously, we get
\begin{align}
  \textstyle
  \left\| \sqrt{\frac{1}{\lambda }-\frac{1}{\nu}}a + \sqrt{\frac{1}{\mu } - \frac{1}{\nu }}b\right\|
  & \textstyle \geq |\alpha| \cdot \|b\| - \beta\|c\|.
\end{align}
This shows \eqref{pr:lem:LB2:toShow1} in the first case.

Assume now that $\max \{\lambda, \mu\} < 2$ and put $ \bar \alpha := \sqrt{\frac{1}{\lambda} - \frac{1}{\nu }} + \sqrt{\frac{1}{\mu }-\frac{1}{\nu}}$. Note that $\bar \alpha \geq |\alpha|$. By again referring to the triangle inequality, we obtain
\begin{align}
  \textstyle
  \left\| \sqrt{\frac{1}{\lambda }-\frac{1}{\nu}}a - \sqrt{\frac{1}{\mu } - \frac{1}{\nu }}b\right\|
  & = \textstyle \left\| \bar\alpha a
  - \sqrt{\frac{1}{\mu } - \frac{1}{\nu }}(a + b) \right\| \nonumber \\
  & \textstyle \geq |\bar \alpha| \cdot \|a\| - \sqrt{\frac{1}{\mu }-\frac{1}{\nu}}\|c\| \nonumber \\
  & \textstyle \geq |\alpha| \cdot \|a\| - \beta\|c\|.
\end{align}
Analogously, we get
\begin{align}
  \textstyle
  \left\| \sqrt{\frac{1}{\lambda }-\frac{1}{\nu}}a - \sqrt{\frac{1}{\mu } - \frac{1}{\nu }}b\right\|
  & \textstyle \geq |\alpha| \cdot \|b\| - \beta\|c\|.
\end{align}
This shows \eqref{pr:lem:LB2:toShow1} in the second case.

On the other hand, using Theorem \ref{t-comp}, we see that the operator $(UT)_{\frac 1 \nu}$ is a cutter with $\operatorname{Fix}UT = \operatorname{Fix}T\cap \operatorname{Fix}U$. Equivalently, $(UT)_{\frac 1 \nu}$ is  a ($-1$)-demicontraction;  see Proposition \ref{prop:demicontractive}. Thus for any $z \in \operatorname{Fix}T\cap \operatorname{Fix}U$, we have
\begin{equation}\label{pr:lem:LB2:c}
  \|c\| = \nu \|(UT)_{\frac 1 \nu}(x) - x\| \leq \nu \|z-x\|.
\end{equation}
Consequently, by combining \eqref{lem:LB1:ineq}, \eqref{pr:lem:LB2:c} and the Cauchy-Schwarz inequality, we get
\begin{align}  \textstyle
  \left\| \sqrt{\frac{1}{\lambda }-\frac{1}{\nu }}a
    \pm \sqrt{\frac{1}{\mu }-\frac{1}{\nu }} b\right\|
    + \beta \|c\|
  & \leq \sqrt{\langle z-x, c\rangle} + \beta \|c\|  \nonumber \\
  & \leq \sqrt{\|z - x\| \cdot \|c\| } + \beta \sqrt{\|c\| \cdot (\nu \|z-x\|)} \nonumber \\
  & = (1 + \beta \sqrt{\nu}) \sqrt{\|z - x\| \cdot \|c\| }.
\end{align}
This and \eqref{pr:lem:LB2:toShow1} yield
\begin{equation}\label{}
  \|c\| \geq \left(\frac{|\alpha|}{1+\beta \sqrt{\nu}}\right)^2 \frac{\max\{\|a\|^2, \|b\|^2\}}{\|z-x\|}.
\end{equation}
In particular, the above-mentioned estimate holds for $z := P_{\operatorname{Fix}T\cap \operatorname{Fix}U}(x)$ in which case $\|z - x\| = d(x, \operatorname{Fix}T\cap \operatorname{Fix}U)$. This completes the proof.
\end{proof}

\bigskip
We are now ready to present our main result. We note that a result analogous to (i) was proved in \cite[Theorem 3.8]{Ceg23}.  We include part (i) for consistency. To the best of our knowledge, parts (ii) and (iii) with $ \max \{\lambda, \mu \} > 2 $ of the following theorem are new.

\begin{theorem}[Regularity] \label{thm:main}
Let $T \colon \mathcal H \to \mathcal H$ and $U \colon \mathcal H \to \mathcal H$ be $\lambda$- and $\mu$-relaxed cutters, respectively, where $\lambda, \mu >0$. Assume that $\lambda \mu < 4$ and $\operatorname{Fix}T\cap \operatorname{Fix}U\neq \emptyset$. Moreover, let $z\in \operatorname{Fix}T\cap \operatorname{Fix}U$, let $r >0$ and put $R :=\max \{r,(\lambda -1)r\}$. The following statements hold:

\begin{enumerate}
\item[$\mathrm{(i)}$] If $T$ is weakly regular over $B(z,r )$ and $U$ is weakly regular over $B(z,R)$, then the operator $UT$ is weakly regular over $B(z,r )$.

\item[$\mathrm{(ii)}$] If $T$ is regular over $B(z, r)$, $U$ is regular over $B(z,R)$ and the family $\{\operatorname{Fix}T,\operatorname{Fix}U\}$ is regular over $B(z,r)$, then the operator $UT$ is regular over $B(z,r)$.

\item[$\mathrm{(iii)}$] If $T$ is $\delta _{1}$-linearly regular over $B(z,r)$, $U$ is $\delta _{2}$-linearly regular over $B(z,R)$ and the family $\{\operatorname{Fix}T,\operatorname{Fix}U\}$ is $\kappa $-linearly regular over $B(z,r)$, where $\delta _{1},\delta _{2} \in (0,1]$ and $\kappa >0$, then the operator $UT$ is $\delta$-linearly regular over $B(z,r)$, where
    \begin{equation} \label{thm:main:delta} \displaystyle
      \delta := \left(\frac{|\alpha|\min \{\delta_1,\delta_2\}} {2 \kappa (1+\beta \sqrt{\nu }) } \right)^2,
    \end{equation}
$\nu$ is given by \eqref{lem:ni:eq1},  and  $\alpha$ and $\beta$ are given by \eqref{lem:LB2:ab}.
\end{enumerate}
\end{theorem}

\begin{proof}
  We begin with some preparatory calculations. First we show that for any $x\in B(z,r )$, we must have
  \begin{equation}\label{pr:thm:main:toShow1}
    T(x) \in B(z, R).
  \end{equation}
  Indeed, by Proposition \ref{prop:demicontractive}, $T$  is a  $\frac{\lambda - 2}{\lambda}$-demicontraction, that is,
  \begin{equation}\label{pr:thm:main:TisDC}
    \|T(x) - z\|^2 \leq \|x - z\|^2 + \frac{\lambda - 2}{\lambda} \|T(x) - x\|^2.
  \end{equation}
  If $\lambda \in (0,2)$, then, by definition, $R = r$. Moreover, $\frac{\lambda - 2}{\lambda} < 0$ and we see that $\|T(x) - z\| \leq \|x-z\| \leq r$. Assume now that $\lambda > 2$. Then $R = (\lambda-1)r > r$.  Since  $T_{\frac{1}{\lambda}}$ is a cutter, it is  a ($-1$)-demicontraction (by  Proposition \ref{prop:demicontractive}) and
  \begin{equation}\label{pr:thm:main:TisRelCutter}
    \|T(x) - x\| = \lambda \|T_{\frac{1}{\lambda}}(x) - x\| \leq \lambda \|z - x\|.
  \end{equation}
  Consequently,  using  \eqref{pr:thm:main:TisDC} and \eqref{pr:thm:main:TisRelCutter}, we get
  \begin{equation}\label{}
    \|T(x) - z\|^2 \leq \|x - z\|^2 + \lambda(\lambda - 2) \|z - x\|^2 \leq (\lambda - 1)^2 r^2 = R^2.
  \end{equation}
  This shows \eqref{pr:thm:main:toShow1}.  We now proceed to each of the individual statements separetely.

  \bigskip
  (i) Let the sequence $\{x^k\}_{k=0}^\infty \subset B(z, r)$ be such that
  \begin{equation}\label{}
    x^{k} \rightharpoonup y \quad \text{and} \quad \|UT(x^{k})-x^{k}\| \to 0
  \end{equation}
  as $k \to \infty$, where $y \in \mathcal H$.  Note that
  \begin{equation}\label{}
    d(x^k, \operatorname{Fix}T \cap \operatorname{Fix} U) \leq \|x^k - z\| \leq r.
  \end{equation}
   Then, thanks to  Lemma \ref{lem:LB2},  we get
  \begin{equation}\label{pr:thm:main:i:assumption}
    \|T(x^{k})-x^{k}\| \to 0 \quad \text{and} \quad \|UT(x^{k})-T(x^{k})\| \to 0.
  \end{equation}
  By assumption, the operator $T$ is weakly regular over $B(z, r)$. This implies that $y \in \operatorname{Fix} T$. On the other hand, we see that $T(x^k) \rightharpoonup y$ as $k \to \infty$ thanks to the first limit of \eqref{pr:thm:main:i:assumption}. Moreover, by \eqref{pr:thm:main:toShow1}, we see that $\{T(x^k)\}_{k=0}^\infty \subset B(z, R)$. By applying the weak regularity of $U$ over $B(z, R)$ to the sequence $\{T(x^k)\}_{k=0}^\infty$, we obtain $y \in \operatorname{Fix} U$. This shows that $y \in \operatorname{Fix} T \cap \operatorname{Fix} U$. Note that $\operatorname{Fix} UT = \operatorname{Fix}T \cap \operatorname{Fix} U$; compare with Theorem \ref{t-comp}. This completes the proof of part (i).

  \bigskip
  (ii) Let the sequence $\{x^k\}_{k=0}^\infty \subset B(z, r)$ be such that
  \begin{equation}\label{pr:thm:main:ii:assumption}
    \|UT(x^{k})-x^{k}\| \to 0
  \end{equation}
  as $k \to \infty$. Analogously to  case (i),  we get
  \begin{equation}\label{pr:thm:main:ii:TandU}
    \|T(x^{k})-x^{k}\| \to 0 \quad \text{and} \quad \|UT(x^{k})-T(x^{k})\| \to 0
  \end{equation}
  as $k \to \infty$, where again $\{T(x^k)\}_{k=0}^\infty \subset B(z, R)$. Applying the regularity of $T$ over $B(z,r)$ and the regularity of $U$ over $B(z, R)$, we get
  \begin{equation}\label{}
    d(x^k, \operatorname{Fix} T) \to 0 \quad \text{and} \quad d(T(x^k), \operatorname{Fix} T) \to 0
  \end{equation}
  as $k \to \infty$. Moreover, the triangle inequality and the definition of the metric projection  yield
  \begin{equation}\label{pr:thm:main:ii:dxk}
    d(x^k, \operatorname{Fix} U) \leq \|x^k - P_{\operatorname{Fix} U}(T(x^k))\| \leq \|x^k - T(x^k)\| + d(T(x^k), \operatorname{Fix} U).
  \end{equation}
  In particular, we obtain
  \begin{equation}\label{}
    \max \{d(x^k, \operatorname{Fix} T),\ d(x^k, \operatorname{Fix} U)\} \to 0
  \end{equation}
  as $k \to \infty$. Using the regularity assumption of the family $\{\operatorname{Fix} T, \operatorname{Fix} U\}$ over $B(z, r)$, we arrive at
  \begin{equation}\label{}
    d(x^k, \operatorname{Fix} T \cap \operatorname{Fix} U) \to 0
  \end{equation}
  as $k \to \infty$. This shows that $UT$ is regular over $B(z, r)$ and completes the proof of part (ii).

  \bigskip
  (iii) Let $x \in B(z, r)$ be such that $x \notin \operatorname{Fix} T \cap \operatorname{Fix} U$. By \eqref{pr:thm:main:toShow1}, we again have $T(x) \in B(z, R)$. Applying the linear regularity of $T$ and $U$ to $x$ and $T(x)$, respectively, we get
  \begin{equation}\label{}
    \|T(x) - x\| \geq \delta_1 d(x, \operatorname{Fix} T) \quad \text{and} \quad
    \|U(T(x)) - T(x)\| \geq \delta_2 d(T(x), \operatorname{Fix} U).
  \end{equation}
  Analogously to \eqref{pr:thm:main:ii:dxk}, we have
  \begin{equation}\label{pr:thm:main:iii:dx}
    d(x, \operatorname{Fix} U) \leq \|x - T(x)\| + d(T(x), \operatorname{Fix} U).
  \end{equation}
  Consequently,
  \begin{align}\label{}
    \max \{\|T(x) - x\|,\  \|UT(x) - T(x)\|\}
    & \textstyle\geq \frac 1 2 (\|T(x) - x\|+ \|UT(x) - T(x)\|)  \nonumber\\
    & \textstyle \geq \frac 1 2 \delta_2 (\|T(x) - x\| + d(T(x), \operatorname{Fix} U)) \nonumber\\
    & \textstyle \geq \frac 1 2 \delta_2 d(x, \operatorname{Fix} U)).
  \end{align}
  Using the linear regularity of the family $\{\operatorname{Fix} T, \operatorname{Fix} U\}$ over $B(z, r)$, we  obtain
  \begin{align}\label{}
    \max \{\|T(x) - x\|,\ & \|UT(x) - T(x)\|\} \nonumber \\
    & \textstyle \geq \frac {\min\{\delta_1, \delta_2\}} {2}
    \max\{ d(x, \operatorname{Fix} T)), d(x, \operatorname{Fix} U))\} \nonumber \\
    & \textstyle \geq \frac {\min\{\delta_1, \delta_2\}} {2 \kappa}
     d(x, \operatorname{Fix} T \cap \operatorname{Fix} U)).
  \end{align}
  By combining this with \eqref{lem:LB2:ineq} of Lemma \ref{lem:LB2}, we arrive at
  \begin{equation}\label{}
    \|UT(x) - x\|
    \geq \left(\frac{|\alpha| \min\{\delta_1, \delta_2\}}{2 \kappa(1+\beta \sqrt{\nu})}\right)^2
    d(x, \operatorname{Fix}T \cap \operatorname{Fix} U).
  \end{equation}
  This shows that $UT$ is $\delta$-linearly regular over $B(z,r)$, which completes the proof
\end{proof}

\begin{remark} \rm
    While estimating the radius $r$ for the regularity of the operators $T$ and $U$ can generally be quite challenging, there are situations in which this regularity holds on $B(z, r)$ for all $r > 0$. This corresponds to the case where the operators $T$ and $U$ are boundedly (weakly/linearly) regular. Some examples of this situation are discussed in Section~\ref{sec:regular_operators}. An analogous statement can be made for a boundedly (linearly) regular pair of sets $\{\operatorname{Fix} T, \operatorname{Fix} U\}$, examples of which are provided in Section~\ref{sec:regular_sets}.

\end{remark}

\begin{remark}[Relaxations Below Two] \label{rem:CRZ18}\rm
  \begin{enumerate}
    \item[(a)] A result analogous to Lemma \ref{lem:LB2} can be found in \cite[Proposition 4.6]{CZ14}. Note, however, that \cite[Proposition 4.6]{CZ14} requires that $\lambda, \mu \in (0,2)$. In such a case, the argument for obtaining \eqref{lem:LB2:ineq} can be made much simpler and the constant on the right-hand side of \eqref{lem:LB2:ineq} becomes $\frac{2-\max \{\lambda ,\mu \}}{\max \{\lambda ,\mu \}}$. We also emphasize that the proofs from \cite{CZ14} do not extend to the case  where  $\max\{\lambda, \mu\} \geq 2$.

    \item[(b)] A result analogous to Theorem \ref{thm:main} can be found in \cite[Theorem 5.4 and Corollary 5.6]{CRZ18},  where the result was expressed in the language of strongly quasi-nonexpansive operators.  In particular, when $\lambda, \mu \in (0,2)$, then $R = r$ and the linear regularity over $B(z,r)$ holds in (iii) with parameter
        \begin{equation}\label{}
          \delta := \frac{2-\max \{\lambda ,\mu \}}{\max \{\lambda ,\mu \}}
          \cdot \left(\frac{\min \{\delta_1,\delta_2\}}{2 \kappa}\right)^2.
        \end{equation}
        Note, however, that the proofs from \cite{CRZ18} rely on \cite[Proposition 4.6]{CZ14}. The proof of Theorem \ref{thm:main} demonstrates the significance of Lemma \ref{lem:LB2} when $ \max \{\lambda, \mu \} \geq 2 $.

    \item[(c)] Other results analogous to Theorem \ref{thm:main} focusing on the preservation of various types of regularity under product and convex combination, where $\lambda, \mu \in (0,2)$, can be found, for example, in \cite{BRZ18, BKRZ19, Ceg15, CRZ20, CZ14, RZ16}. See also \cite{CET21}.

  \end{enumerate}
\end{remark}

\begin{corollary}[Regularity for Projections] \label{cor:main3}
Let $A$ and $B$ be closed and convex subsets of $\mathcal H$. Assume that $\lambda, \mu >0$, $\lambda \mu < 4$ and that $A \cap B \neq \emptyset$. Let $z \in A \cap B$ and let $r > 0$. The following statements hold:

\begin{enumerate}
\item[$\mathrm{(i)}$] The operator $(P_B)_\mu (P_A)_\lambda$ is weakly regular over $B(z,r)$.

\item[$\mathrm{(ii)}$] If $\{A, B\}$ is regular over $B(z, r)$, then the operator $(P_B)_\mu (P_A)_\lambda$ is regular over $B(z,r)$.

\item[$\mathrm{(iii)}$] If $\{A, B\}$ is $\kappa $-linearly regular over $B(z,r)$, then the operator $(P_B)_\mu (P_A)_\lambda$ is $\delta$-linearly regular over $B(z,r)$, where

    \begin{equation} \label{cor:main3:delta} \displaystyle
      \delta := \left(\frac{|\alpha|  \min\{\lambda, \mu\}} {2 \kappa (1+\beta \sqrt{\nu }) } \right)^2,
    \end{equation}
$\nu$ is given by \eqref{lem:ni:eq1}, $\alpha$ and $\beta$ are given by \eqref{lem:LB2:ab}.
\end{enumerate}
\end{corollary}

\begin{remark}[Generalized Douglas-Rachford Operator] \label{rem:GDRO} \rm
  Let $A$ and $B$ be closed and convex subsets of $\mathcal{H}$ with $A \cap B \neq \emptyset$, and let $\lambda, \mu > 0$. Consider the following operator:
  \begin{equation}\label{}
    V := (1 - \bar \alpha) I + \bar \alpha \big( (P_B)_\mu (P_A)_\lambda - I\big),
  \end{equation}
  where $\bar\alpha \in (0, \infty)$; see, for example, \cite{DP18}. The operator $V$ reduces to the well-known Douglas-Rachford operator when $\lambda = \mu = 2$ and $\bar \alpha = \frac{1}{2}$. Following \cite{DP18}, we call $V$ the generalized Douglas-Rachford operator.
  \begin{enumerate}
    \item[(a)] Assume that $\lambda \mu < 4$. It is not difficult to prove that the (weak/linear) regularity of $V$ over $B(z, r)$ is equivalent to the (weak/linear) regularity of $(P_B)_\mu (P_A)_\lambda$, respectively. Consequently, with minor adjustments, we can reformulate Corollary~\ref{cor:main3} for the operator $V$. In particular, Corollary~\ref{cor:main3}(iii) recovers part of \cite[Proposition 4.9(i)]{DP18} and \cite[Proposition 4.14]{DP18}, where $\lambda, \mu \in (0,2]$ and $\min \{\lambda, \mu\} < 2$. In these results, the pair of sets $\{A, B\}$ is transversal (see Example~\ref{ex:RegSets}(v)) and boundedly linearly regular, respectively.

    \item[(b)] When $\lambda = \mu = 2$, both operators correspond to reflections, and Corollary~\ref{cor:main3} no longer applies to $V$. However, an independent argument presented in \cite[Theorem 4.4]{BNP15} shows that when $\mathcal{H}$ is finite-dimensional and the family of sets $\{A, B\}$ is transversal, then the Douglas-Rachford operator $V$ is boundedly linearly regular.
  \end{enumerate}
\end{remark}

\section[An Application]{An  Application\label{sec:application}}

In this section we apply the results of Section \ref{sec:main} to the convergence
properties  of method  \eqref{int:xk}, where we make the range for $\alpha_k$ more precise. We note that a result analogous to (i) was shown in \cite[Theorem 4.5]{Ceg23}.  We include part (i) for consistency. To the best of our knowledge, parts (ii) and (iii) are new when $\max \{\lambda, \mu\} > 2$.

\begin{theorem}[Convergence] \label{thm:main2}
Let $T \colon \mathcal H \to \mathcal H$ and $U \colon \mathcal H \to \mathcal H$ be $\lambda$- and $\mu$-relaxed cutters, respectively, where $\lambda, \mu >0$. Assume that $\lambda \mu < 4$ and $\operatorname{Fix}T\cap \operatorname{Fix}U\neq \emptyset$. Consider the following iterative method:
\begin{equation}\label{thm:main2:xk}
  x^0 \in \mathcal H, \quad x^{k+1} := x^k + \frac{\alpha_k}{\nu} \big(UT(x^k) - x^k \big), \quad k = 0,1,2,\ldots,
\end{equation}
where $\alpha_k \in [\varepsilon, 2 - \varepsilon]$, $\varepsilon > 0$, and where $\nu$ is given by \eqref{lem:ni:eq1}. Moreover, let $z\in \operatorname{Fix}T\cap \operatorname{Fix}U$, let $r := \|x^0 - z\| >0$ and put $R :=\max \{r,(\lambda -1)r\}$. The following statements hold:

\begin{enumerate}
\item[$\mathrm{(i)}$] If $T$ is weakly regular over $B(z,r )$ and $U$ is weakly regular over $B(z,R)$, then the sequence $\{x^k\}_{k=0}^\infty$ converges weakly to some $x^* \in \operatorname{Fix}T\cap \operatorname{Fix}U$.

\item[$\mathrm{(ii)}$] If $T$ is regular over $B(z, r)$, $U$ is regular over $B(z,R)$ and the family $\{\operatorname{Fix}T,\operatorname{Fix}U\}$ is regular over $B(z,r)$, then the sequence $\{x^k\}_{k=0}^\infty$ converges in norm to some $x^* \in \operatorname{Fix}T\cap \operatorname{Fix}U$.

\item[$\mathrm{(iii)}$] If $T$ is $\delta _{1}$-linearly regular over $B(z,r)$, $U$ is $\delta _{2}$-linearly regular over $B(z,R)$ and the family $\{\operatorname{Fix}T,\operatorname{Fix}U\}$ is $\kappa $-linearly regular over $B(z,r)$, where $\delta _{1},\delta _{2} \in (0,1]$ and $\kappa >0$, then the sequence $\{x^k\}_{k=0}^\infty$ converges $Q$-linearly to some $x^* \in \operatorname{Fix}T\cap \operatorname{Fix}U$ at the rate $\sqrt{1 - \left(\frac{\varepsilon \delta}{2\nu}\right)^2}$, where $\delta$ is given by \eqref{thm:main:delta}, that is,
    \begin{equation}\label{thm:main2:rate}
      \|x^{k+1} -  x^*\| \leq \sqrt{1 - \left(\frac{\varepsilon \delta}{2\nu}\right)^2} \ \|x^k -  x^*\|, \quad k = 0,1,2,\ldots.
    \end{equation}
\end{enumerate}
\end{theorem}

\begin{proof}
  Let $V := I + \frac{1}{\nu}(UT - I)$ and put $V_k := I + \alpha_k(V - I)$.  By Theorem \ref{t-comp}, the operator $V$ is a cutter with $\operatorname{Fix} V = \operatorname{Fix}T\cap \operatorname{Fix}U$. Thus, Proposition \ref{prop:demicontractive} yields that $V_k$ is $\rho_k$-demicontractive with the same set of fixed points, where $\rho_k := \frac{\alpha_k-2}{\alpha_k} < 0$. Noting that $x_{k+1} = V_k(x^k)$ and $x^{k+1} - x^k = \frac{\alpha_k}{\nu} (UT(x^k) - x^k)$, for each $w \in \operatorname{Fix}T\cap \operatorname{Fix}U$, we get
  \begin{equation}\label{pr:thm:main2:DC}
    \|x^k - w\|^2 - \|x^{k+1} - w\|^2
    \geq \frac{2-\alpha_k}{\alpha_k}\|x^{k+1} - x^k\|^2
    \geq \left(\frac{\varepsilon}{\nu}\right)^2 \|UT(x^k) - x^k\|^2.
  \end{equation}
  Consequently, $\{x^k\}_{k=0}^\infty $ is Fej\'{e}r monotone with respect to $\operatorname{Fix}T\cap \operatorname{Fix}U$ and
  \begin{equation}\label{pr:thm:main2:eq1}
    \|UT(x^k) - x^k\| \to 0
  \end{equation}
  as $k \to \infty$. Moreover, for $w = z$, we get $\{x^k\}_{k=0}^\infty \subset B(z, r)$.

  \bigskip
  (i) Using  Theorem \ref{thm:main} (i),  we see that  the operator $UT$ is weakly regular over $B(z,r)$. This and \eqref{pr:thm:main2:eq1} imply that every weak cluster point of $\{x^k\}_{k=0}^\infty$ lies in $\operatorname{Fix}T\cap \operatorname{Fix}U$. Applying Lemma \ref{lem:FM} (i),  we complete  the proof of this part.

  \bigskip
  (ii) Theorem \ref{thm:main} (ii) yields that the operator $UT$ is regular over $B(z,r)$. Thanks to  \eqref{pr:thm:main2:eq1}, we see that
  \begin{equation}\label{pr:thm:main2:BR}
    d(x^k, \operatorname{Fix}T\cap \operatorname{Fix}U) \to 0
  \end{equation}
  as $k \to \infty$. It now suffices to apply Lemma \ref{lem:FM} (ii).

  \bigskip
  (iii) Thanks to Theorem \ref{thm:main} (iii), the product $UT$ is $\delta$-linearly regular over $B(z,r)$, where $\delta >0$ is defined in \eqref{thm:main:delta}. In particular,
  \begin{equation}\label{pr:thm:main2:LR}
    \|UT(x^k) - x^k\| \geq \delta d(x^k, \operatorname{Fix}T\cap \operatorname{Fix}U).
  \end{equation}
  On the other hand, linear regularity implies regularity over $B(z,r)$. Thus, thanks to (ii), we know that $x^k \to x^*$ for some $x^* \in \operatorname{Fix}T\cap \operatorname{Fix}U$. Applying Lemma \ref{lem:FM} (iii), we get
  \begin{equation}\label{pr:thm:main2:LB} \textstyle
    d(x^k, \operatorname{Fix}T\cap \operatorname{Fix}U) \geq \frac 1 2 \|x^k - x^*\|
  \end{equation}
  for $k = 0, 1,2, \ldots$. By combining \eqref{pr:thm:main2:DC} (with $w = x^*$), \eqref{pr:thm:main2:LR} and \eqref{pr:thm:main2:LB}, we arrive at
  \begin{equation}\label{}
    \|x^{k+1} - x^*\| \leq \sqrt{1 - \left(\frac{\varepsilon \delta}{2\nu}\right)^2} \|x^k - x^*\|.
  \end{equation}
  In particular, $\frac{\varepsilon \delta}{2\nu} < 1$.  This completes the proof.
\end{proof}

\begin{remark}[Reformulation of \eqref{thm:main2:xk}] \rm \label{rem:refForXk}
Assume that $\lambda, \mu > 0$ and $\lambda \mu < 4$. It is sometimes more convenient to write method \eqref{thm:main2:xk} as
\begin{equation}\label{rem:refForXk:xk}
  x^0 \in \mathcal{H}, \quad x^{k+1} := x^k + \bar{\alpha}_k \big( UT(x^k) - x^k \big), \quad k = 0,1,2,\ldots,
\end{equation}
where
\begin{equation}\label{rem:refForXk:alphak}
  \bar{\alpha}_k \in [\bar{\varepsilon}, 1 + \rho - \bar{\varepsilon}], \quad \bar{\varepsilon} > 0, \quad \text{and} \quad
  \rho := \frac{2-\nu}{\nu}.
\end{equation}
Note that methods \eqref{thm:main2:xk} and \eqref{rem:refForXk:xk} are equivalent (we demonstrate this below). Moreover, since $\lambda \mu < 4$, we obtain
\begin{equation}\label{rem:refForXk:rho}
  -1 < \rho =
  \begin{cases}
    \left( \frac{\lambda}{2-\lambda} + \frac{\mu}{2- \mu}\right)^{-1}, & \text{if } \lambda \neq 2 \text{ and } \mu \neq 2,\\
    0, & \text{if } \lambda = 2 \text{ or } \mu = 2,
  \end{cases}
\end{equation}
see \eqref{lem:ni:eq1}. An upper bound for $\bar{\alpha}_k$ analogous to \eqref{rem:refForXk:alphak}, with $\rho$ given by \eqref{rem:refForXk:rho}, can be found, for example, in \cite[Theorem 2.14 and Corollary 5.12]{DP18}.
\end{remark}

\begin{proof}
  We show the equivalence of methods \eqref{thm:main2:xk} and \eqref{rem:refForXk:xk}. Indeed, we write
  \begin{equation}
    2 - \varepsilon = \nu + 2 - \nu - \varepsilon = \left( 1 + \frac{2-\nu}{\nu} - \frac{\varepsilon}{\nu} \right) \nu = (1 + \rho - \frac{\varepsilon}{\nu}) \nu.
  \end{equation}
  If $\{x^k\}_{k=0}^\infty$ is defined by \eqref{thm:main2:xk} with $\alpha_k \leq 2-\varepsilon$, then $\bar{\alpha}_k := \frac{\alpha_k}{\nu} \leq 1 + \rho - \bar{\varepsilon}$, where $\bar{\varepsilon} = \frac{\varepsilon}{\nu}$. On the other hand, if $\{x^k\}_{k=0}^\infty$ is defined by \eqref{rem:refForXk:xk} with $\bar{\alpha}_k \leq 1 + \rho - \bar{\varepsilon}$, then $\alpha_k := \bar{\alpha}_k \nu \leq (2 - \varepsilon) \nu$, where $\varepsilon = \bar{\varepsilon} \nu$.
\end{proof}

\begin{corollary}[Convergence for Projections] \label{cor:main4}
Let $A$ and $B$ be closed and convex subsets of $\mathcal H$. Assume that $\lambda, \mu >0$, $\lambda \mu < 4$ and that $A \cap B \neq \emptyset$. Consider the following iterative method:
\begin{equation}\label{cor:main4:xk}
  x^0 \in \mathcal H, \quad x^{k+1} := x^k + \frac{\alpha_k}{\nu} \Big((P_B)_\mu (P_A)_\lambda(x^k) - x^k \Big), \quad k = 0,1,2,\ldots,
\end{equation}
where $\alpha_k \in [\varepsilon, 2 - \varepsilon]$, $\varepsilon > 0$, and where $\nu$ is given by \eqref{lem:ni:eq1}. Moreover, let $z\in A \cap B$ and let $r := \|x_0 - z\| >0$. The following statements hold:

\begin{enumerate}
\item[$\mathrm{(i)}$] The sequence $\{x^k\}_{k=0}^\infty$ converges weakly to some $x^* \in A \cap B$.

\item[$\mathrm{(ii)}$] If $\{A, B\}$ is regular over $B(z, r)$, then the sequence $\{x^k\}_{k=0}^\infty$ converges in norm to some $x^* \in A \cap B$.

\item[$\mathrm{(iii)}$] If $\{A, B\}$ is $\kappa $-linearly regular over $B(z,r)$, then the sequence $\{x^k\}_{k=0}^\infty$ converges $Q$-linearly to some $x^* \in A \cap B$ at the rate $\sqrt{1 - \left(\frac{\varepsilon \delta}{2\nu}\right)^2}$, where this time $\delta$ is given by \eqref{cor:main3:delta}.
\end{enumerate}
\end{corollary}

\begin{remark}[Generalized Douglas-Rachford Method] \label{rem:GRDM} \rm
  Let $A$ and $B$ be closed and convex subsets of $\mathcal{H}$ with $A \cap B \neq \emptyset$, and let $\lambda, \mu > 0$. Following \cite{DP18}, we consider the following iterative method:
  \begin{equation}\label{rem:GRDM:xk}
    x^0 \in \mathcal{H}, \quad x^{k+1} := x^k + \bar{\alpha} \Big( (P_B)_\mu (P_A)_\lambda(x^k) - x^k \Big), \quad k = 0,1,2,\ldots,
  \end{equation}
  where $ \bar{\alpha} \in (0, \infty)$.
  \begin{enumerate}
    \item[(a)] Assume that $\lambda \mu < 4$ and $\bar{\alpha} \in [\bar \varepsilon, 1 + \rho - \bar{\varepsilon}]$, where $\rho$ is given by \eqref{rem:refForXk:rho} and $\bar \varepsilon > 0$. Clearly, method \eqref{rem:GRDM:xk} is a special case of \eqref{cor:main4:xk}, as discussed in Remark \ref{rem:refForXk}. Consequently, we can apply Corollary \ref{cor:main4} to obtain (weak/linear) convergence. In particular, Corollary \ref{cor:main4}(i) recovers part of \cite[Theorem 2.14]{DP18}, where $\lambda, \mu \in (0,2]$ and $\min \{\lambda, \mu\} < 2$. Moreover, Corollary \ref{cor:main4}(iii) recovers \cite[Corollary 5.12(ii)(c)]{DP18}.

    \item[(b)] It should be noted that when $\lambda = \mu = 2$ (both operators are reflections), Corollary \ref{cor:main4} no longer applies to method \eqref{rem:GRDM:xk}. However, this case is considered in \cite[Theorem 2.14 and Corollary 5.12]{DP18}, and also in \cite[Theorem 8.5]{BNP15}.
  \end{enumerate}
\end{remark}

\subsection{Numerical Example}

In this section, we consider a very simple numerical example illustrating Corollary \ref{cor:main4}. We assume that $\mathcal{H} = \mathbb{R}^2$ so that we can visualize the trajectories of the considered methods. We also assume that $A$ and $B$ are two lines intersecting at the origin $(0,0)$, with the angle $\frac{\pi}{6}$. Thus, the unique solution of the linear feasibility problem is the point $(0,0)$. We examine the following three iterative methods:

\begin{itemize}
  \item The method of alternating projections (MAP), where
  \begin{equation}\label{numExp:MAP}
    x^0 \in \mathbb{R}^2, \quad x^{k+1} := P_B P_A(x^k), \quad k = 0,1,2,\ldots.
  \end{equation}

  \item The Douglas-Rachford method (DR), where
  \begin{equation}\label{numExp:DR}
    x^0 \in \mathbb{R}^2, \quad x^{k+1} := x^k + \frac{1}{2} \Big( (P_B)_2 (P_A)_2(x^k) - x^k \Big), \quad k = 0,1,2,\ldots.
  \end{equation}

  \item Method \eqref{cor:main4:xk} with $\lambda = 1$, $\mu = 3$ (so that $\nu = 4$), and $\alpha_k = 1$, that is,
  \begin{equation}\label{numExp:NEW}
    x^0 \in \mathbb{R}^2, \quad x^{k+1} := x^k + \frac{1}{4} \Big( P_B (P_A)_3(x^k) - x^k \Big), \quad k = 0,1,2,\ldots.
  \end{equation}
\end{itemize}

In all three methods \eqref{numExp:MAP}, \eqref{numExp:DR}, and \eqref{numExp:NEW}, we use the starting point $x_0 = (1,0)$ and calculate 30 iterations. In Figure \ref{figure1}, we visualize the trajectories $\{x^k\}_{k=0}^{30}$ of the generated iterates. In Figure \ref{figure2}, we show the absolute errors $\{\log \|x_k\|\}_{k=0}^{30}$ obtained by the corresponding methods. Figure \ref{figure1} demonstrates quite interesting behavior of the third method. Figure \ref{figure2} verifies linear convergence. Both figures suggest that the projection methods with relaxation parameters $(\lambda, \mu)$ satisfying $\lambda \mu < 4$ can also be considered in a more advanced numerical study.

\begin{figure}[h]
  \centering
  \includegraphics[scale=0.7]{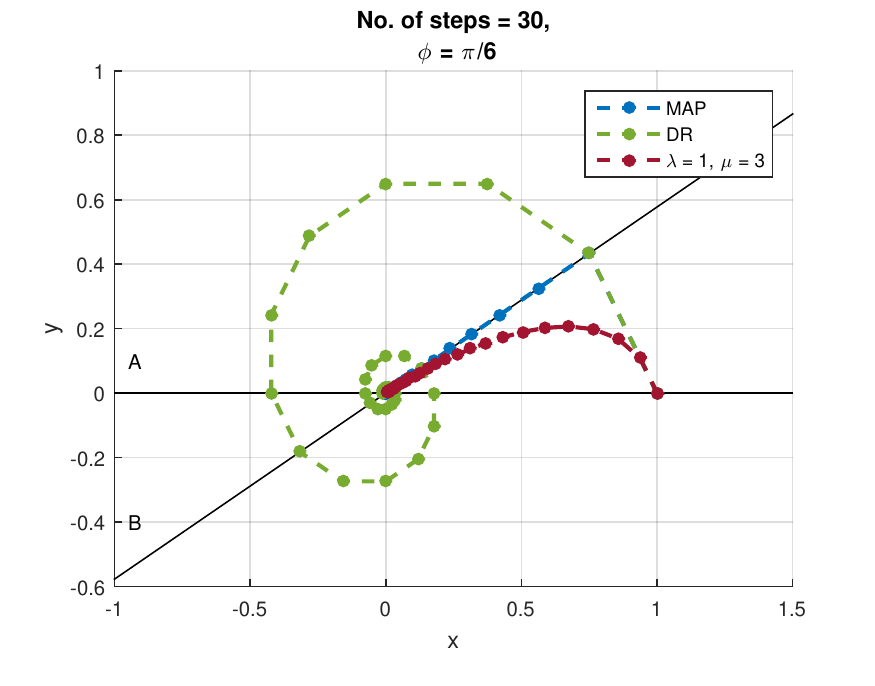}
  \caption{Trajectories $\{x^k\}_{k=0}^{30}$ generated by methods \eqref{numExp:MAP}, \eqref{numExp:DR} and \eqref{numExp:NEW}.}\label{figure1}
\end{figure}

\begin{figure}[h]
  \centering
  \includegraphics[scale=0.7]{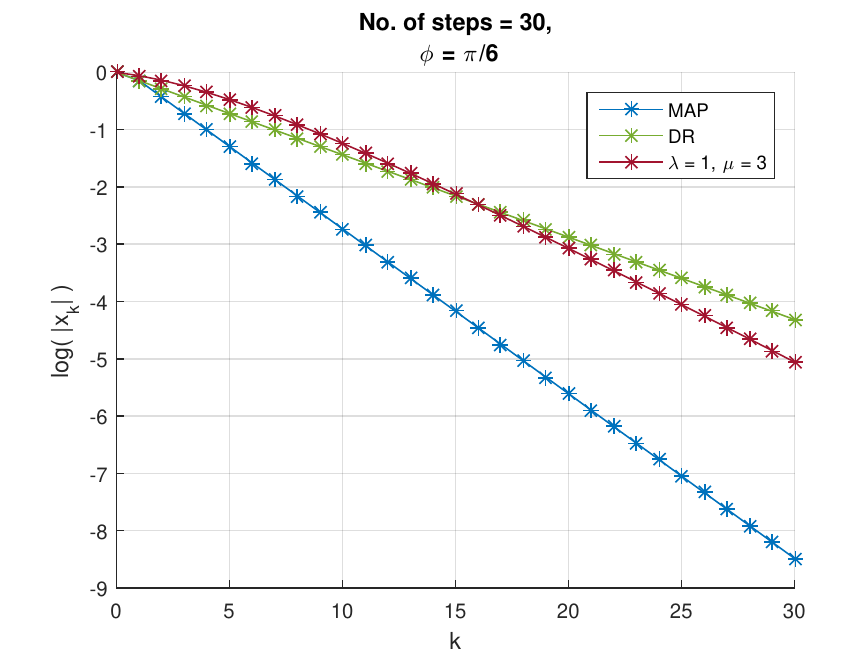}
  \caption{Absolute errors $\{\log \|x_k\|\}_{k=0}^{30}$ obtained by methods \eqref{numExp:MAP}, \eqref{numExp:DR} and \eqref{numExp:NEW}.}\label{figure2}
\end{figure}

\bigskip
\noindent\textbf{Author Contributions.} All the authors contribute equally to this work.

\bigskip
\noindent\textbf{Acknowledgement.} We are grateful to two anonymous referees for their close reading of our manuscript, and for their detailed comments and constructive suggestions.

\bigskip
\noindent\textbf{Funding.} This research was partially supported by the Israel Science Foundation (Grant 820/17), by the Fund for the Promotion of Research at the Technion (Grant 2001893) and by the Technion General Research Fund (Grant 2016723).

\bigskip
\noindent\textbf{Availability of Supporting Data.} Data sharing is not applicable to this article as no data sets were generated or analysed during the current study.

\section*{Declarations}
\addcontentsline{toc}{section}{Declarations}

\textbf{Competing interests.}  The authors declare that they have no conflict of interest.

\bigskip

\noindent\textbf{Ethical Approval.} Not applicable.

\end{document}